\documentclass[a4paper,12pt]{article}
\usepackage{amsmath}
\usepackage{amssymb}
\usepackage{amsthm}

\theoremstyle{plain}
\newtheorem{theorem}{Theorem}[section]

\newtheorem{corollary}{Corollary}[section]
\theoremstyle{definition}

\begin{document}

 \title{The arithmetic derivative and Leibniz-additive functions}

\author{
{ \sc Pentti Haukkanen \& Jorma K. Merikoski} \\ 
Faculty of Natural Sciences, FI-33014 University of Tampere\\ Finland\\
 pentti.haukkanen@uta.fi \& jorma.merikoski@uta.fi\\
{\sc Timo Tossavainen} \\
Department of Arts, Communication and Education, \\
Lulea University of Technology, SE-97187 Lulea, Sweden
\\ timo.tossavainen@ltu.se\\
}

\maketitle

 \begin{abstract}
An arithmetic function $f$ is Leibniz-additive if there is a completely multiplicative function $h_f$, i.e., $h_f(1)=1$ and $h_f(mn)=h_f(m)h_f(n)$ for all positive integers $m$ and $n$, satisfying
$$
f(mn)=f(m)h_f(n)+f(n)h_f(m)
$$
for all positive integers $m$ and $n$. A motivation for the present study is the fact that Leibniz-additive functions are generalizations of the arithmetic derivative $D$; namely, $D$ is Leibniz-additive with $h_D(n)=n$. 
In this paper, we study the basic properties of Leibniz-additive functions and, among other things, show that a Leibniz-additive function $f$ is totally determined by the values of $f$ and $h_f$ at primes. We also consider properties of Leibniz-additive functions with respect to the usual product, composition and Dirichlet convolution of arithmetic functions. 
\end{abstract}

\section{Introduction}
 Let $n$ be a positive integer. Its {\it arithmetic derivative} $D(n)=n'$ is defined
 as follows: 

\medskip
\noindent
(i)\hskip1cm $p'=1$ for all primes $p$,

\medskip
\noindent
(ii)\hskip0.8cm $(mn)'=mn'+m'n$ for all positive integers $m$ and $n$.

\medskip
\noindent

Given
$$n=\prod_{q\in\mathbb{P}}q^{\nu_q(n)},$$ 
where $\mathbb{P}$ is the set of primes,
the formula for computing the arithmetic derivative of $n$ is (see, e.g., \cite{Ba, UA})
$$n'=n\sum_{p\in\mathbb{P}}\frac{\nu_p(n)}{p}.$$ 
A brief summary on the history of arithmetic derivative and its generalizations to other number sets can be found, e.g., in \cite{Ba, UA, HMT}. 

Similarly, one can define {\it the arithmetic partial derivative} (see, e.g., \cite{Ko, HMT}) via
$$
D_p(n)=n_p'=\frac{\nu_p(n)}{p} n,
$$ 
and {\it the arithmetic logarithmic derivative} \cite{UA} as
$$
 {\rm ld}(n)=\frac{D(n)}{n}.
$$

An arithmetic function $f$ is said to be {\it additive} if $f(mn)=f(m)+f(n)$, whenever gcd$(m, n)=1$, and {\it multiplicative} if $f(1)=1$ and $f(mn)=f(m)f(n)$, whenever gcd$(m, n)=1$. Additive and multiplicative functions are totally determined by their values at prime powers. Further,
an arithmetic function $f$ is said to be {\it completely additive} if $f(mn)=f(m)+f(n)$ for all positive integers $m$ and $n$, and 
{\it completely multiplicative} if $f(1)=1$ and $f(mn)=f(m)f(n)$ for all positive integers $m$ and $n$. Completely additive and completely multiplicative functions are totally determined by their values at primes. 

\begin{theorem}\label{cmca}
Let 
\begin{equation}\label{fta}
n=q_1q_2\cdots q_r=p_1^{n_1}p_2^{n_2}\cdots p_s^{n_s}, 
\end{equation}
where  $q_1, q_2,\ldots, q_r$ are primes and $p_1, p_2,\ldots, p_s$ are distinct primes. 
If $f$ is completely additive, then $f(1)=0$ and 
$$
f(n)=\sum_{i=1}^r f(q_i)=\sum_{i=1}^s n_i f(p_i), 
$$
and if $f$ is completely multiplicative, then $f(1)=1$ and  
$$
f(n)=\prod_{i=1}^r f(q_i)=\prod_{i=1}^s f(p_i)^{n_i}.  
$$
\end{theorem} 

The functions defined above are widely studied in the literature, see, e.g., \cite{Ap,KM,LT,Mc,SC,Sc, Sh}. 

\bigskip

We say that an arithmetic function $f$ is {\em Leibniz-additive} (or, {\em L-additive}, in short) if there is a completely multiplicative function $h_f$ such that 
\begin{equation}\label{gca}
f(mn)=f(m)h_f(n)+f(n)h_f(m)
\end{equation}
for all positive integers $m$ and $n$. 
Then $f(1)=0$ since $h_f(1)=1$. 
The property \eqref{gca} may be considered a generalized Leibniz rule.
For example, the arithmetic derivative $D$ is L-additive with $h_D(n)=n$, 
since it satisfies the usual Leibniz rule 
$$
D(mn)=D(m)n+D(n)m
$$
for all positive integers $m$ and $n$, and the function $h_D(n)=n$ is completely multiplicative. Similarly, the arithmetic partial derivative respect to the prime $p$ is L-additive with $h_{D_p}(n)=n$.
Further, all completely additive functions $f$ are L-additive with $h_f(n)=1$. For example, the logarithmic derivative of $n$ is completely additive since
$$
 {\rm ld}(mn) = {\rm ld}(m)+{\rm ld}(n).
$$

The term ``L-additive function" seems to be new in the literature, yet Chawla \cite{Ch} has defined  the concept of completely distributive arithmetic function meaning the same as we do with an L-additive function. However, this is a somewhat misleading term since a distributive arithmetic function usually refers to a property that 
\begin{equation}\label{distr}
f(u\ast v)=(fu)\ast(fv),
\end{equation}
 i.e., the function $f$ distributes over the Dirichlet convolution. This is satisfied by completely multiplicative arithmetic functions, not by completely distributive functions as Chawla defined them. 

Because L-additivity is analogous with generalized additivity and generalized multiplicativity (defined in \cite{Ha}), we could, alternatively, speak about generalized complete additivity (and also define the concept of generalized complete multiplicativity). 

In this paper, we consider L-additive functions especially from the viewpoint that they are generalizations of the arithmetic derivative. In the next section, we present their basic properties. In the last section, we study L-additivity and the arithmetic derivative in terms of the Dirichlet convolution.


\section{Basic properties}

\begin{theorem}\label{quotient}
Let $f$ be an arithmetic function. If $f$ is L-additive and
$h_f$ is nonzero-valued, then $f/h_f$ is completely additive.
Conversely, if there is a completely multiplicative nonzero-valued function~$h$
such that $f/h$ is completely additive, then $f$ is L-additive and $h_f=h$.
\end{theorem}

\begin{proof} 
If $f$ satisfies~\eqref{gca} and $h_f$ is never zero, then
$$
\frac{f(mn)}{h_f(mn)}=\frac{f(m)h_f(n)+f(n)h_f(m)}{h_f(m)h_f(n)}=
\frac{f(m)}{h_f(m)}+\frac{f(n)}{h_f(n)},
$$
verifying the first part. The second part follows by substituting $h_f=h$
in the above and exchanging the sides of the second equation. 
\end{proof} 

\begin{theorem}\label{prod}
If $f$ is  L-additive such that $h_f(n)\ne 0$ for all  positive integers $n$, then 
$$
f=g_f h_f,
$$
where $g_f$ is completely additive. Conversely, if $f$ is of the form
\begin{equation}\label{fact}
f=gh, 
\end{equation}
where $g$ is completely additive and $h$ is completely multiplicative, then $f$ is L-additive with $h_f=h$.
\end{theorem}

\begin{proof}
If $f$ is  L-additive such that $h_f(n)\ne 0$ for all  positive integers $n$, 
then denoting $g_f=f/h_f$ in Theorem~\ref{quotient} we obtain $f=g_f h_f$, 
 where $g_f$ is completely additive. 

Conversely, assume that $f$ is of the form \eqref{fact}. Then 
\begin{eqnarray*}
f(mn)
&=&h(m)h(n)[g(m)+g(n)]=(hg)(m)h(n)+(hg)(n)h(m)
\\
&=&f(m)h(n)+f(n)h(m).
\end{eqnarray*}
\end{proof} 

\begin{corollary}
If an arithmetic function $f$ satisfies the Leibniz rule
$$
 f(mn) = f(m)n + f(n)m
$$
for all positive integers $m, n$, then $f(n) = g_f(n)n$, where $g_f$ is completely additive. Also the converse holds.
\end{corollary} 

\begin{proof} 
The function $f$ is L-additive with $h_f(n) = n$. 
\end{proof} 

Theorem \ref{prod} shows that each L-additive function $f$ such that $h_f$ is always nonzero can be represented as a pair of a completely additive function $g_f$ and a completely multiplicative function $h_f$. In this case, we write $f=(g_f, h_f)$. 
However, if $h_f(p)=0$ for some prime $p$ and $f(p)$ is nonzero, then there is not any function $g_f$ such that $f=g_f h_f$ and, consequently, a representation of this kind is not possible. 

On the other hand, the representation of a L-additive function as such a pair is not always unique either. For instance, if $f$ and $g_f$ are identically zero, then $h_f$ may be any completely multiplicative function. 
Next theorem will however show that this representation is unique whenever $f$ and $h_f$ are nonzero at all primes. Observe that the latter condition is equivalent to assuming that $h_f$ is nonzero for all positive integers, cf. Theorem 2.2.

\begin{theorem}
If $f$ is L-additive such that $f(p)\ne 0$ and $h_f(p)\ne 0$ for all primes $p$, then the representation $f=(g_f, h_f)$ is unique. 
\end{theorem}

\begin{proof} 
Let $f=(g_f, h_f)=(\tilde{g}_f, \tilde{h}_f)$. Then for all primes $p$,  
$$
f(p)=g_{f}(p)h_{f}(p)=\tilde{g}_f(p)\tilde{h}_f(p). 
$$
On the other hand, by the definition of L-additivity,
$$
f(p^2)=2f(p)h_{f}(p)=2f(p)\tilde{h}_f(p), 
$$
which implies that $h_{f}(p)=\tilde{h}_f(p)$ and, consequently, $g_f(p)=\tilde{g}_f(p)$. 
\end{proof} 

In particular, $D = (g_D, h_D)$, where $g_D(p) =1/p = {\rm ld}(p)$ and $h_D(p) = p$ for all primes $p$. In other words, the arithmetic derivative has the representation as the pair of the logarithmic derivative and the identity function.
Similarly, for the arithmetic partial derivative respect to the prime $p$, we have $D_{p} = (g_{D_{p}}, h_{D_{p}})$, where $g_{D_{p}}(q)=0$ if $q\ne p$, $g_{D_{p}}(p)=1/p$, and $h_{D_{p}}(q) = q$ for all primes $q$.

\medskip
We saw in Theorem \ref{cmca} that if $f$ is completely additive or completely multiplicative, then it is totally defined by its values at primes. If $f$ is only L-additive, then we must also know the values of $h_f$ at primes.

\begin{theorem}\label{totally}
Let $n$ be as in \eqref{fta}. If $f$ is L-additive, then 
$$
f(n)
= \sum_{i=1}^r 
h_f(q_1)\cdots h_f(q_{i-1})f(q_i)h_f(q_{i+1})\cdots h_f(q_r). 
$$
If $h_f(p_1),\dots,h_f(p_s)\ne 0$, then
$$
f(n)=h_f(n)\sum_{i=1}^r\frac{f(q_i)}{h_f(q_i)}=h_f(n)\sum_{i=1}^s\frac{n_i f(p_i)}{h_f(p_i)}. 
$$

\end{theorem}

\begin{proof}
In order to verify the first claim, it suffices to notice that  
\begin{eqnarray*}
f(n) &=& f(q_1)h_f(q_2\cdots q_r)+f(q_2\cdots q_r)h_f(q_1)\\
&=& f(q_1)h_f(q_2)\cdots h_f(q_r)+f(q_2\cdots q_r)h_f(q_1)\\
&=& f(q_1)h_f(q_2)\cdots h_f(q_r)
+h_f(q_1) \big[ f(q_2)h_f(q_3\cdots q_r)+f(q_3\cdots q_r)h_f(q_2) \big]\\
&=& f(q_1)h_f(q_2)\cdots h_f(q_r)
+h_f(q_1) \big[ f(q_2)h_f(q_3)\cdots h_f(q_r)+f(q_3\cdots q_r)h_f(q_2) \big]\\
&=& \cdots\\
&=&\sum_{i=1}^r h_f(q_1)\cdots h_f(q_{i-1})f(q_i)h_f(q_{i+1})\cdots h_f(q_r). 
\end{eqnarray*}

The rest of the theorem follows from the equation $h_f(n)=h_f(q_1)\cdots h_f(q_r)$, which holds since $h_f$ is completely multiplicative. 
\end{proof}
\medskip

Since $D(p) = 1$ and $h_D(p) = p$ for all primes $p$, this theorem gives the well-known formula
$$
D(n)=n\sum_{i=1}^r\frac{1}{q_i}
=\sum_{i=1}^r q_1\cdots q_{i-1}q_{i+1}\cdots q_r
=n\sum_{i=1}^s\frac{n_i}{p_i}. 
$$
It also implies that 
$$
f(p^k)=kh_f(p)^{k-1} f(p)
$$
for all primes $p$ and nonnegative integers $k$. (For $k=0$, $h_f(p)$ must be nonzero.) In particular, for the arithmetic derivative, this reads
$$
D(p^k)=kp^{k-1}. 
$$

\begin{theorem}
If $u$ is L-additive and $v$ is completely multiplicative, then their product function $uv$ is L-additive with $h_{uv} = h_uv$. In particular, the function $f(n) = D(n)n^k$, where k is a nonnegative integer, is L-additive with $h_f(n)=n^{k+1}$. 
\end{theorem}

\begin{proof} 
For all positive integers $m$ and $n$, 
\begin{eqnarray*}
(uv)(mn)&=&u(mn)v(mn)=\bigl(u(m)h_u(n)+u(n)h_u(m)\bigr)v(m)v(n)\\
&=&(uv)(m)(h_u v)(n)+(uv)(n)(h_u v)(m). 
\end{eqnarray*}
Thus $uv$ is L-additive with completely multiplicative part equaling to $h_u v$. 
\end{proof} 

\begin{theorem}\label{circ}
If $v$ is L-additive and $u$ is completely multiplicative with positive integer values, then their composite function $v\circ u$ is L-additive with $h_{v\circ u} = h_v\circ u$ .
\end{theorem}

\begin{proof} 
For all positive integers $m$ and $n$, 
\begin{eqnarray*}
(v\circ u)(mn)&=&v(u(mn))=v(u(m)u(n))=v(u(m))h_v(u(n))+v(u(n))h_v(u(m))\\
&=&(v\circ u)(m)(h_v\circ u)(n)+(v\circ u)(n)(h_v\circ u)(m). 
\end{eqnarray*}
It remains to show that $h_v\circ u$ is completely multiplicative. We have 
$$
(h_v\circ u)(mn)=h_v(u(mn))=h_v(u(m)u(n))=h_v(u(m))h_v(u(n))=(h_v\circ u)(m)(h_v\circ u)(n). 
$$
This completes the proof. 
\end{proof} 

Let $u$ be as in Theorem \ref{circ}. Since $D$ is L-additive, then, by this theorem, $D\circ u$ is L-additive with $h_{D\circ u}=u$; that is
$$
 D(u(mn)) = u(n)D(u(m))+u(m)D(u(n))
$$
 for all positive integers $m, n$. In particular,
$$
 D((mn)^k) = n^kD(m^k)+m^kD(n^k),
$$
 where $k$ is a nonnegative integer. These equations follow also from the fact that $D$ satisfies the Leibniz rule.


\section{L-additive functions in terms of the Dirichlet convolution}

Above we have seen that many fundamental properties of the arithmetic derivative, e.g., the formula for computing the arithmetic derivative of a given positive integer,
are rooted to the fact that $D$ is L-additive. We complete this article by changing our point of view slightly and demonstrate that L-additive functions can also be studied in terms of the Dirichlet convolutions. 

\medskip

Let $u$ and $v$ be arithmetic functions. Their {\it Dirichlet convolution} is
$$
(u\ast v)(n)=\sum_{\substack{a,b=1\\ab=n}}^n u(a)v(b).
$$
We let $f(u\ast v)$ denote the product function of $f$ and $u\ast v$, i.e.,
$$
(f(u\ast v))(n) = f(n)(u\ast v)(n).
$$

\begin{theorem}
An arithmetic function $f$ is completely additive if and only if
$$
f(u\ast v) = (fu)\ast v+u\ast (fv)
$$
for all arithmetic functions $u$ and $v$.
\end{theorem}

\begin{proof} 
See \cite[Proposition 2]{Sc}. 
\end{proof} 

Next theorems shows that L-additive functions can  also be characterized in an analogous way. 

\begin{theorem} 
Let $f$ be an arithmetic function. If $f$ is L-additive and $h_f$ is
nonzero-valued, then 
\begin{eqnarray}
\label{gen_Schw}
f(u\ast v) = (fu)\ast (h_fv)+(h_fu)\ast (fv)
\end{eqnarray}
for all arithmetic functions $u$ and $v$. Conversely, if there is a completely
multiplicative nonzero-valued function~$h$ such that
\begin{eqnarray}
\label{convassump}
f(u\ast v) = (fu)\ast (hv)+(hu)\ast (fv)
\end{eqnarray}
for all arithmetic functions $u$ and $v$, then $f$ is L-additive and $h_f=h$.
\end{theorem}

\begin{proof} 
Under the assumptions of the first part, Theorems~2.1 and~3.1 imply
$$
(f/h_f)(u\ast v)=(fu/h_f)\ast v+u\ast (fv/h_f).
$$
Multiplying by~$h_f$, the left-hand side becomes $f(u\ast v)$. Since
$$
h_f((fu/h_f)\ast v)=(fu)\ast (h_fv),\quad h_f(u\ast (fv/h_f))=(h_fu)\ast (fv)
$$
by \eqref{distr}, the claim follows. To prove the second part, multiply~(\ref{convassump})
by~$1/h$ and perform a simple modification of the above.
\end{proof} 

\begin{corollary} 
If $u$ and $v$ are arithmetic functions, then
$$
D(u\ast v)=(Du)\ast (Nv)+(Nu)\ast (Dv),
$$
where $N(n)=n$.
\end{corollary}

\begin{proof} 
It suffices to notice that $h_D=N$. 
\end{proof} 

On the other hand, taking $u=v$,  equation \eqref{gen_Schw} becomes 
$$
f(u^{\ast 2})=f(u\ast u)=2[(fu)\ast (h_f u)].    
$$
For $u=E$, where $E(n)=1$ for all positive integers $n$, this reads  
$$
f\tau=2(f\ast h_f),   
$$ 
where $\tau$ is the divisor-number-function. 
Especially, we have 
$$
D(u^{\ast 2})=2[(Du)\ast (Nu)]   
$$
and, with $u=E$,  
$$
D\tau=2(D\ast N).    
$$ 

\bigskip

Some remarks can be made also in the opposite direction.  Assume now that 
$$
f\tau=2(f\ast h)   
$$ 
for some completely multiplicative function $h$ that is nonzero for all positive integers $n$. 
Then 
$$
(f/h)\tau=2((f/h)\ast E).   
$$ 
Thus, again according to \cite[Proposition 2]{Sc}, we see that  $f/h$ is completely additive, which shows that 
$f$ is L-additive with $h_f=h$. 
In particular, if 
$$
f\tau=2(f\ast N),  
$$ 
then $f$ is L-additive with $h_f=N$. For example, $D$ satisfies this condition.

\medskip

Further properties of L-additive functions in terms of the Dirichlet convolution can be derived from the results in \cite{Sc, LT}. It would be possible to obtain some properties of L-additive functions in terms of the unitary convolution as well from the results in \cite{LT}.

\end{document}